\documentclass[11pt]{amsart}
\usepackage{amssymb}
\usepackage{mathrsfs}
\usepackage{graphicx}
\usepackage{tabularx}
\usepackage{setspace}
\usepackage{longtable}
\usepackage{geometry}
\usepackage{cite}
\usepackage{mathtools}
\usepackage{bm}
\usepackage{enumitem}
\usepackage{hyperref}
\usepackage{subcaption}
\usepackage{tikz}
\usepackage{tikz-cd}

\scshape
\date{}
\newtheorem{thm}{\bf Theorem}[section]
\newtheorem{cor}[thm]{\bf Corollary}
\newtheorem{lem}[thm]{\bf Lemma}

\newtheorem{defn}[thm]{\bf Definition}

\newtheorem{rem}[thm]{\bf Remark}
\newtheorem{exam}[thm]{\bf Example}

\newcommand{\Irr}{\mathrm{Irr}}
\newcommand{\cd}{\mathrm{cd}}
\usepackage{amsmath,amssymb}

\begin{document}

\title[Splitting fields and spectral invariants]{Splitting fields and spectral invariants of character degree graphs in solvable groups}

\author [G. Sivanesan]{G. Sivanesan}
\address{Department of Mathematics, Government College of Engineering, Salem 636011, Tamil Nadu, India, ORCID: 0000-0001-7153-960X.}
\email{sivanesan@gcesalem.edu.in}

\author [C. Selvaraj]{C. Selvaraj}
\address{Department of Mathematics, Periyar University, Salem 636011, Tamil Nadu, India,  ORCID:  0000-0002-4050-3177.}
\email{selvavlr@yahoo.com}

\author[Laubacher]{Jacob Laubacher}
\address{Department of Mathematics, Hillsdale College, Hillsdale, Michigan 49242, USA, ORCID:0000-0003-0045-7951.}
\email{jlaubacher@hillsdale.edu}

\keywords{character degree graph, splitting field, solvable group, eigenvalue, Galois group, regular graph, irreducibility, spectral graph theory\\\indent\emph{Corresponding author.} Jacob Laubacher :\href{mailto:jlaubacher@hillsdale.edu}{jlaubacher@hillsdale.edu}}


\begin{abstract}
In this paper, we investigate the eigenvalues of character degree graphs, 
with particular emphasis on the arithmetic properties of their spectra. 
First, we study \((n-2)\)-regular character degree graphs of solvable groups 
and derive an explicit formula for their characteristic polynomials. We show 
that all their eigenvalues are rational and, consequently, that their splitting 
field is \(\mathbb{Q}\). We then consider supergraphs obtained by adding edges 
to these graphs and prove that the corresponding splitting field is a quadratic 
extension of \(\mathbb{Q}\).

Next, using their structural decomposition, we examine a general class of Lewis 
graphs. For this class, we establish bounds on both the number of irrational 
eigenvalues and the degree of the associated splitting fields. Finally, we 
investigate prime character degree graphs of diameter \(3\), focusing on the 
arithmetic nature of their eigenvalues and the degree of their splitting fields.
\end{abstract}

\date{}
\maketitle

\section{Introduction}

Let $G$ be a finite solvable group. Denote by $\Irr(G)$ the set of all irreducible complex characters of $G$, and define the set of character degrees by
\[
\cd(G)=\{\chi(1)\mid \chi\in\Irr(G)\}.
\]
Let $\rho(G)$ be the set of all prime divisors of the degrees in $\cd(G)$. The \emph{character degree graph} of $G$, denoted by $\Delta(G)$, is the simple undirected graph with vertex set $\rho(G)$, where two distinct primes $p$ and $q$ are adjacent if and only if $pq$ divides $\chi(1)$ for some $\chi\in\Irr(G)$.

The study of character degree graphs was initiated by Manz \emph{et al.} and has since become an important tool in the study of finite solvable groups \cite{ManzStaszewski1988}. Fundamental results show that $\Delta(G)$ has at most two connected components \cite{Manz1985}, and Manz and Wolf~\cite{ManzWolf1989} proved that the sum of the diameters of these components is at most three. A comprehensive overview of character degree graphs appears in the survey by Lewis~\cite{Lewis2008}. Further classification approaches were developed by Bissler \emph{et al.}~\cite{Bissler2019b,Bissler2025}, DeGroot \emph{et al.}~\cite{LJ}, and Laubacher \emph{et al.}~\cite{LM, LMS}, who investigated families of graphs realizable as $\Delta(G)$.

A central problem is to determine which graphs arise as character degree graphs of solvable groups. Important advances include Hamiltonian character degree graphs \cite{Ebrahimi2015}, Eulerian character degree graphs \cite{Sivanesan2024}, metric dimension investigations \cite{Cameron2025}, and graphs containing cut vertices \cite{Hafezieh2021}. These results illustrate that the structure of $\Delta(G)$ is closely tied to the algebraic properties of $G$.

Spectral graph theory has recently provided new insight into character degree graphs. Eigenvalues of the adjacency matrix encode structural information such as connectivity, diameter, and regularity. Ebrahimi \emph{et al.}~\cite{Ebrahimi2023} studied regular character degree graphs with spectra contained in $[-2,\infty)$, while Hafezieh \emph{et al.}~\cite{Hafezieh2021} analyzed relations between eigenvalues and cut vertices.

Let $A=A(\Delta(G))$ denote the adjacency matrix of $\Delta(G)$ and
\[
P(\lambda)=\det(\lambda I-A)
\]
its characteristic polynomial. We investigate the \emph{splitting field} of $P(\lambda)$ over $\mathbb{Q}$, which measures the algebraic complexity of the spectrum. Field-theoretic methods have been applied successfully to integral graphs \cite{Harary1974}, circulant graphs \cite{Monius2022}, and normal Cayley graphs \cite{Wu2024}.

In this paper we compute characteristic polynomials for several families of character degree graphs, determine their splitting fields.

\section{Notation and Preliminaries}

We adopt the following notation throughout the paper.

\begin{itemize}
\item $G$ always denotes a finite solvable group.
\item $\Gamma$ denotes a finite simple graph.
\end{itemize}

All graphs considered are finite and simple. The eigenvalues of a graph $\Gamma$ are the eigenvalues of its adjacency matrix $A(\Gamma)$, and the multiset of eigenvalues is called the \emph{spectrum}, denoted by $\mathrm{spec}(\Gamma)$ (see~\cite{Cvetkovic1995,Cvetkovic2010}).

Let $H$ be a graph with vertex set $\{v_1,\ldots,v_k\}$ and let $\mathcal{F}=\{\Gamma_1,\ldots,\Gamma_k\}$ be a family of graphs. The \emph{$H$-join} $\bigvee_H\mathcal{F}$ replaces each vertex $v_i$ with $\Gamma_i$ and connects every vertex of $\Gamma_i$ to every vertex of $\Gamma_j$ whenever $v_iv_j\in E(H)$.

A generalization, called the \emph{$H$-generalized join constrained by vertex subsets} \cite{Cardoso2013a,Cardoso2013b}, uses subsets $\mathcal{S}=\{S_1,\ldots,S_k\}$ with $S_i\subseteq V(\Gamma_i)$, where edges are added only between vertices of $S_i$ and $S_j$ whenever $v_iv_j\in E(H)$.

If all $\Gamma_i=\Gamma$, this construction reduces to the lexicographic product $H[\Gamma]$~\cite{Godsil2001}.

\subsection*{Generalization of Fiedler’s Lemma}

\begin{rem}[Notation Convention]
The symbol $\rho(G)$ denotes the set of prime divisors of character degrees. To avoid conflict with this standard notation, the coupling parameters appearing in Fiedler's lemma and its generalizations \cite{Fiedler1974,Cardoso2011} are denoted by $\gamma$ and $\gamma_j$ instead of the original symbol $\rho$ used in those results. This modification is purely notational.
\end{rem}

\begin{lem}[Fiedler~\cite{Fiedler1974}]
Let $A\in\mathbb{R}^{m\times m}$ and $B\in\mathbb{R}^{n\times n}$ be symmetric matrices with eigenvalues $\alpha_1,\ldots,\alpha_m$ and $\beta_1,\ldots,\beta_n$, and unit eigenvectors $u,v$ corresponding to $\alpha_1,\beta_1$. Then
\[
C=
\begin{bmatrix}
A & \gamma\,uv^T\\
\gamma\,vu^T & B
\end{bmatrix}
\]
has eigenvalues $\alpha_2,\ldots,\alpha_m, \beta_2,\ldots,\beta_n$ together with the eigenvalues of
\[
\widehat C=
\begin{bmatrix}
\alpha_1 & \gamma\\
\gamma & \beta_1
\end{bmatrix}.
\]
\end{lem}

\begin{lem}[Generalized Version~\cite{Cardoso2011}]
Let $A_j\in\mathbb{R}^{n_j\times n_j}$ be symmetric matrices with orthonormal eigenvectors $\mathbf u_{ij}$ and eigenvalues $\alpha_{ij}$. Then
\[
C=
\begin{pmatrix}
A_1 & \gamma_1u_{11}u_{12}^T & &\\
\gamma_1u_{12}u_{11}^T & A_2 & \ddots &\\
& \ddots & \ddots & \gamma_{k-1}u_{1(k-1)}u_{1k}^T\\
& & \gamma_{k-1}u_{1k}u_{1(k-1)}^T & A_k
\end{pmatrix}
\]
has spectrum
\[
\bigcup_{j=1}^k
\left(\sigma(A_j)\setminus\{\alpha_{1j}\}\right)
\cup
\sigma(\widehat C),
\]
where $\widehat C$ is tridiagonal with diagonal entries
$\alpha_{1j}$ and off-diagonal entries $\gamma_j$.
\end{lem}

\begin{thm}{\cite{Cardoso2011}}\label{1.1}
Let $\Gamma_j$ be $p_j$-regular graphs on $n_j$ vertices. Define
\[
\widetilde{\Gamma}
=\bigoplus_{j=1}^{k}\Gamma_j,
\qquad
\gamma_j=\sqrt{n_jn_{j+1}}.
\]
Then
\[
\sigma(\widetilde{\Gamma})
=
\left(
\bigcup_{j=1}^{k}
(\sigma(\Gamma_j)\setminus\{p_j\})
\right)
\cup
\sigma(\widehat C).
\]
\end{thm}

\begin{rem}\label{rem:method}
All spectral results in this paper follow from Theorem~\ref{1.1}. Whenever a graph admits a generalized join decomposition
\[
\widetilde{\Gamma}=\bigoplus_{j=1}^{k}\Gamma_j,
\]
its spectrum is obtained by combining the nontrivial eigenvalues of the component graphs $\Gamma_j$ with the eigenvalues of the quotient matrix $\widehat C$. Hence eigenvalue computations reduce to determining $\sigma(\widehat C)$.
\end{rem}

\begin{rem}{\cite{Lewis2008}}\label{rem:question}
Let $n$ and $N$ be integers greater than $1$ with $N \ge 2^{\,n}-1$. It is natural to ask whether there exists a solvable group $G$ whose character degree graph $\Delta(G)$ has exactly two connected components of cardinalities $n$ and $N$, respectively.

At present, there is no known pair $(n,N)$ for which such a solvable group fails to exist. In particular, for every positive integer $N$, one can construct a solvable group $G$ such that $\Delta(G)$ has exactly two connected components, one of which is an isolated vertex and the other containing $N$ vertices.

We briefly outline a standard construction illustrating this phenomenon. Let $p$ be an odd prime. By Zsigmondy’s prime theorem, there exists an integer $a$ such that $p^{a}-1$ has at least $N$ distinct prime divisors. Choose a divisor $b$ of $p^{a}-1$ with exactly $N$ distinct prime divisors, and let $E$ be an extra-special $p$-group of order $p^{2a+1}$ and exponent $p$. Then $E$ admits an automorphism $\sigma$ of order $b$ that centralizes $Z(E)$. Setting $G = E \rtimes \langle \sigma \rangle$, one obtains
\[
\mathrm{cd}(G) = \{1,\, b,\, p^{a}\}.
\]
Consequently, the connected components of $\Delta(G)$ are precisely $\{p\}$ and $\pi(b)$. Thus $\Delta(G)$ has two connected components, one consisting of a single vertex and the other consisting of $N$ vertices.
\end{rem}

\begin{lem}\label{lem:two-components}
\textnormal{\cite[Lemma~2.6]{Lewis2008}}
For every positive integer $N$, there exists a solvable group $G$ such that the character degree graph $\Delta(G)$ has exactly two complete, connected components: one component consists of a single isolated vertex, and the other component has $N$ vertices.
\end{lem}

\begin{defn}\cite[p. 502]{Bissler2019a}\label{direct product} Using direct product, one can construct higher order character degree graphs. For two groups $A$ and $B$ where $\rho(A)$ and  $\rho(B)$ are disjoint,  we have that $\rho(A\times B) =  \rho(A) \cup \rho(B)$.  Define an edge between vertices $p$ and $q$ in $\rho(A\times B)$  if any of the following is satisfied:
\begin{itemize}
\item [\rm (i)] $p , q\in \rho (A)$ and there is an edge between $p$ and $q$ in $\Delta(A);$
\item [\rm (ii)] $p, q \in \rho (B)$ and there is an edge between $p$ and $q$ in $\Delta(B);$
\item [\rm (iii)] $p\in\rho (A)$ and $q\in\rho (B);$
\item [\rm (iv)]  $p \in\rho (B) $ and $q\in\rho (A).$ 
\end{itemize}
Now we get a higher order character degree graph and it is called direct product and the same is denoted by $\Delta(A\times B).$
\end{defn}

\begin{thm}[P\'alfy’s Three Prime Theorem~\cite{Palfy1998}]
For any three distinct vertices in \( \Delta(G) \), at least one edge connects two of them. In particular, \( \Delta(G) \) has at most two connected components.
\end{thm}

\begin{thm}{\cite{Ebrahimi2015}} \label{1.2}
If \( \Delta(G) \) is not a block and has diameter at most 2, then every block is complete.
\end{thm}

\begin{thm}{\cite{LewisMeng2019}} \label{1.3}
For solvable \( G \), the graph \( \Delta(G) \) has at most one cut vertex.
\end{thm}

\begin{thm}{\cite{Zuccari2014}}\label{1.4}
If \( \Delta(G) \) is a non-complete regular graph of a solvable group with \( n \) vertices, it is \( (n-2) \)-regular.
\end{thm}

\begin{thm}{\cite{Hafezieh2021}} \label{1.5}
Let \( G \) be a solvable group such that \( \Delta(G) \) is a regular graph of order \( n \). Then \( \Delta(G) \) has three distinct eigenvalues if and only if the following equivalent conditions hold:
\begin{enumerate}[label=(\roman*)]
    \item \( n \geq 4 \) is even and \( \Delta(G) = K_n - M \), where \( M \) is a perfect matching. In particular, \( \Delta(G) \) is \( (n - 2) \)-regular.
    \item \( G \) is the direct product of at least two groups with disconnected character graphs of two vertices, such that the prime divisors of character degrees of distinct factors of \( G \) are disjoint.
\end{enumerate}
\end{thm}

\begin{thm}{\cite{Lewis2006}}\label{1.6}
A graph \( \Gamma \) on \( n \) vertices is \( \Delta(G) \) for some solvable group \( G \) of Fitting height \(2\) if and only if the vertices of degree less than \( n-1 \) can be partitioned into two complete subgraphs, one of which consists entirely of vertices of degree \( n-2 \).
\end{thm}

Motivated by the characterization from above, Peter J. Cameron referred to a graph having the following properties as a \textbf{Lewis graph} in \cite{Cameron2025}. A graph \( \Gamma \) on \( n \) vertices is called a \emph{Lewis graph} if its set of vertices of degree less than \( n-1 \) can be partitioned into two complete subgraphs, one of which consists entirely of vertices of degree \( n-2 \).

Let the vertex set of \( \Gamma \) be partitioned as follows:
\begin{itemize}[label=--]
    \item \( U \): the set of \textbf{universal vertices}, that is, vertices of degree \( n-1 \), with \( |U|=n_1 \);
    \item \( X \): the set of non-universal vertices of degree \( n-2 \), with \( |X|=n_2 \);
    \item \( Y \): the remaining non-universal vertices, with \( |Y|=n_3 \).
\end{itemize}
Thus,
\[
n=n_1+n_2+n_3.
\]

By the characterization of a Lewis graph, the induced subgraphs \( \Gamma[U] \), \( \Gamma[X] \), and \( \Gamma[Y] \) are complete graphs. In the complement graph \( \overline{\Gamma} \), the vertices in \( U \) are isolated. Hence the only possible missing edges of \( \Gamma \) occur between vertices of \( X \) and vertices of \( Y \).

The adjacency pattern between \( X \) and \( Y \) may be encoded by a matrix
\[
B_{XY}\in \{0,1\}^{n_2\times n_3}.
\]
Here \( B_{ij}=1 \) indicates that \( x_i\in X \) is adjacent to \( y_j\in Y \), whereas \( B_{ij}=0 \) indicates that the edge \( x_i y_j \) is missing in \( \Gamma \).

\begin{thm}[Gauss’s Lemma {\cite{DummitFoote}}]\label{1.7}
Let \( f(x)\in\mathbb{Z}[x] \) be a \emph{primitive} polynomial. Then \( f(x) \) is irreducible in \( \mathbb{Z}[x] \) if and only if it is irreducible in \( \mathbb{Q}[x] \).
\end{thm}

\section{Splitting Fields of Character Degree Graphs for Solvable Groups}

Ebrahimi \emph{et al.}~\cite{Ebrahimi2023} studied regular character degree graphs whose adjacency spectra lie in $[-2,\infty)$, obtaining strong structural restrictions. From a related perspective, Hafezieh \emph{et al.}~\cite{Hafezieh2021} showed that any $(n-2)$-regular character degree graph of a finite solvable group has exactly three distinct adjacency eigenvalues.

The present theorem sharpens these results for solvable groups. We prove that for even $n\ge4$, an $(n-2)$-regular character degree graph is uniquely determined up to isomorphism and must be the complete multipartite graph
\[
K_{2,2,\dots,2}\quad\left(\text{with } \tfrac{n}{2}\text{ parts of size }2\right).
\]
Moreover, we explicitly compute its adjacency spectrum and characteristic polynomial, and show that the associated splitting field over $\mathbb{Q}$ is trivial.

Finally, these hypotheses are natural and non-vacuous: the complement of $K_{2,2,\dots,2}$ is a triangle-free perfect matching, satisfying Pálfy’s condition, and such graphs are known to occur as character degree graphs of solvable groups of Fitting height two (see Lewis~\cite{Lewis2006}). Thus the theorem provides a complete spectral and algebraic description of this class of character degree graphs.

\begin{thm}\label{thm:char_poly_deltaG}
Let \(n\ge4\) be an even integer and let \(G\) be a solvable group whose character degree graph \(\Delta(G)\) is \((n-2)\)-regular. Then the characteristic polynomial of \(\Delta(G)\) is
\[
\chi(x)=x^{\frac{n}{2}}(x+2)^{\frac{n}{2}-1}(x-(n-2)).
\]
Moreover, the splitting field of \(\chi(x)\) over \(\mathbb{Q}\) is \(\mathbb{Q}\), and hence the extension degree is \(1\).
\end{thm}

\begin{proof}
Let \(G\) be a solvable group such that the character degree graph \(\Delta(G)\) is \((n-2)\)-regular on \(n\) vertices, where \(n\ge4\) is even. We first determine the structure of \(\Delta(G)\).

Since every vertex of \(\Delta(G)\) has degree \(n-2\), the degree of each vertex in the complement graph \(\overline{\Delta(G)}\) is
\[
(n-1)-(n-2)=1.
\]
Hence \(\overline{\Delta(G)}\) is a \(1\)-regular graph on \(n\) vertices and therefore a disjoint union of edges. Because \(n\) is even, these edges cover all vertices, and thus \(\overline{\Delta(G)}\) is a perfect matching.

Consequently, the vertex set of \(\Delta(G)\) admits a partition
\[
V(\Delta(G))=V_1\sqcup V_2\sqcup\cdots\sqcup V_{n/2},
\qquad |V_i|=2,
\]
such that the two vertices of each \(V_i\) are nonadjacent in \(\Delta(G)\), while every vertex of \(V_i\) is adjacent to every vertex of \(V_j\) for \(i\neq j\). Hence each induced subgraph \(\Delta(G)[V_i]\) is an edgeless graph on two vertices, that is,
\[
\Delta(G)[V_i]\cong E_2 \qquad (1\le i\le n/2),
\]
and the bipartite subgraph between \(V_i\) and \(V_j\) is complete whenever \(i\neq j\).

Thus \(\Delta(G)\) is the complete multipartite graph
\[
K_{2,2,\dots,2},
\]
with \(\frac{n}{2}\) parts of size \(2\). Equivalently,
\[
\Delta(G)\cong K_{n/2}[E_2,E_2,\dots,E_2],
\]
the \(H\)-join obtained by replacing each vertex of \(K_{n/2}\) with a copy of the edgeless graph \(E_2\).

Having established this structure, we compute the adjacency spectrum using the \(H\)-join theorem from spectral graph theory. Let \(k=\frac{n}{2}\). In the representation
\[
\Delta(G)\cong K_k[E_2,E_2,\dots,E_2],
\]
each \(G_i=E_2\) is \(0\)-regular on \(n_i=2\) vertices. The spectrum of \(E_2\) is
\[
\sigma(E_2)=\{0^{(2)}\}.
\]
Removing one copy of the regular eigenvalue \(0\) from each \(E_2\) contributes \(k=\frac{n}{2}\) eigenvalues equal to \(0\).

The remaining eigenvalues are obtained from the matrix \(M=[m_{ij}]\in\mathbb{R}^{k\times k}\) defined by
\[
m_{ii}=r_i,\qquad 
m_{ij}=
\begin{cases}
\sqrt{n_in_j}, & ij\in E(K_k),\\
0, & \text{otherwise}.
\end{cases}
\]
Since \(r_i=0\) and \(n_i=2\) for all \(i\), we obtain \(m_{ii}=0\) and \(m_{ij}=2\) for \(i\neq j\). Hence
\[
M=2A(K_k),
\]
where \(A(K_k)\) denotes the adjacency matrix of the complete graph \(K_k\).

The adjacency spectrum of \(K_k\) is
\[
\{(k-1)^1,\,(-1)^{k-1}\},
\]
and therefore
\[
\sigma(M)=\{(2k-2)^1,\,(-2)^{k-1}\}.
\]
Since \(k=\frac{n}{2}\), this becomes
\[
\sigma(M)=\{(n-2)^1,\,(-2)^{\frac{n}{2}-1}\}.
\]

Combining these eigenvalues with the \(k=\frac{n}{2}\) zeros obtained earlier, the adjacency spectrum of \(\Delta(G)\) is
\[
\sigma(\Delta(G))=\{(n-2)^1,\,(-2)^{\frac{n}{2}-1},\,0^{\frac{n}{2}}\}.
\]
Hence the characteristic polynomial of \(\Delta(G)\) is
\[
\chi(x)=(x-(n-2))(x+2)^{\frac{n}{2}-1}x^{\frac{n}{2}}.
\]

Finally, the roots \(0\), \(-2\), and \(n-2\) are all rational numbers. Therefore the splitting field of \(\chi(x)\) over \(\mathbb{Q}\) is \(\mathbb{Q}\) itself, and the extension degree is
\[
[\mathbb{Q}:\mathbb{Q}]=1.
\]
This completes the proof. \hfill $\square$
\end{proof}

\begin{exam}\label{ex:n8-regular}
Let $n=8$. Then $k=\frac{n}{2}=4$, and the character degree graph $\Delta(G)$ is isomorphic to the H-join
\[
\Delta(G) \cong K_4[E_2, E_2, E_2, E_2],
\]
that is, the H-join of four copies of the edgeless graph $E_2$ with respect to the complete graph $K_4$. Since each copy of $E_2$ consists of two isolated vertices, the total number of vertices is $4 \times 2 = 8$. For every edge $ij \in E(K_4)$, the H-join construction inserts all possible edges between the corresponding $E_2$-parts, producing a complete bipartite graph $K_{2,2}$ between those parts. Consequently, each vertex is adjacent to all vertices lying in the other three parts, that is, adjacent to $3 \times 2 = 6$ vertices, and hence $\Delta(G)$ is a $6$-regular graph on $8$ vertices, which is an $(n-2)$-regular graph for $n = 8$.

\begin{center}
    \begin{tikzpicture}[scale=1.3, every node/.style={circle, draw, inner sep=2pt}]


        \coordinate (A) at ({2.5*cos(112.5)}, {2.5*sin(112.5)}); 
        \coordinate (B) at ({2.5*cos(67.5)},  {2.5*sin(67.5)});  
        \coordinate (C) at ({2.5*cos(157.5)}, {2.5*sin(157.5)}); 
        \coordinate (D) at ({2.5*cos(202.5)}, {2.5*sin(202.5)}); 
        \coordinate (E) at ({2.5*cos(247.5)}, {2.5*sin(247.5)}); 
        \coordinate (F) at ({2.5*cos(292.5)}, {2.5*sin(292.5)}); 
        \coordinate (G) at ({2.5*cos(337.5)}, {2.5*sin(337.5)}); 
        \coordinate (H) at ({2.5*cos(22.5)},  {2.5*sin(22.5)});  

        \node[label=above left:$u_1$]  (nA) at (A) {};
        \node[label=above right:$u_2$] (nB) at (B) {};
        \node[label=left:$u_3$]        (nC) at (C) {};
        \node[label=left:$u_4$]        (nD) at (D) {};
        \node[label=below left:$u_5$]  (nE) at (E) {};
        \node[label=below right:$u_6$] (nF) at (F) {};
        \node[label=right:$u_7$]       (nG) at (G) {};
        \node[label=right:$u_8$]       (nH) at (H) {};


        \draw[thick] (nA)--(nC); \draw[thick] (nA)--(nD);
        \draw[thick] (nB)--(nC); \draw[thick] (nB)--(nD);

        \draw[thick] (nA)--(nE); \draw[thick] (nA)--(nF);
        \draw[thick] (nB)--(nE); \draw[thick] (nB)--(nF);

        \draw[thick] (nA)--(nG); \draw[thick] (nA)--(nH);
        \draw[thick] (nB)--(nG); \draw[thick] (nB)--(nH);

        \draw[thick] (nC)--(nE); \draw[thick] (nC)--(nF);
        \draw[thick] (nD)--(nE); \draw[thick] (nD)--(nF);

        \draw[thick] (nC)--(nG); \draw[thick] (nC)--(nH);
        \draw[thick] (nD)--(nG); \draw[thick] (nD)--(nH);

        \draw[thick] (nE)--(nG); \draw[thick] (nE)--(nH);
        \draw[thick] (nF)--(nG); \draw[thick] (nF)--(nH);

    \end{tikzpicture}
    
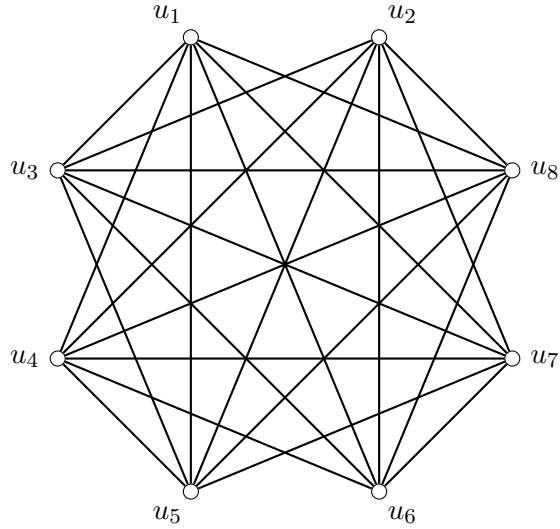
\captionof{figure}{An $(n-2)$-regular character degree graph
                       $\Delta(G) \cong K_4[E_2,E_2,E_2,E_2]$ for $n=8$.}
    \label{fig:n8-regular}
\end{center}

The graph depicted in Figure~\ref{fig:n8-regular} is the unique $(n-2)$-regular character degree graph on eight vertices, which arises for solvable groups as described in the general classification (see Lewis \emph{et al.}~\cite[p.~38]{Lewis2026}, graph labeled $D_{96}$ therein).

By Theorem~\ref{thm:char_poly_deltaG}, the characteristic polynomial of the adjacency matrix of $\Delta(G) \cong K_4[E_2, E_2, E_2, E_2]$ is
\[
\chi(x) = (x - 6)(x + 2)^3\, x^4.
\]
Accordingly, the adjacency spectrum of $\Delta(G)$ consists of:
\begin{itemize}
    \item the eigenvalue $6 = n - 2$ with multiplicity $1$ (the largest eigenvalue, equal to the regularity degree),
    \item the eigenvalue $-2$ with multiplicity $3 = k - 1$, arising from the complete graph $K_4$ factor of the H-join, and
    \item the eigenvalue $0$ with multiplicity $4 = k$, each contributed by one copy of $E_2$ in the H-join construction.
\end{itemize}
Since all eigenvalues are integers, $\chi(x)$ splits completely over $\mathbb{Q}$, and hence the splitting field of $\chi(x)$ is $\mathbb{Q}$ itself, with extension degree $[\mathbb{Q} : \mathbb{Q}] = 1$.
\end{exam}

\begin{rem}{\cite{Lewis2001}}\label{rem:general-structure-bridge}
Let $G$ be a finite solvable group and suppose that the set of primes dividing the irreducible character degrees of $G$ admits a partition
\[
\rho(G)=\pi_1 \,\dot\cup\, \pi_2 \,\dot\cup\, \{p\},
\]
where $\pi_1$ and $\pi_2$ are nonempty and disjoint, and $p\notin \pi_1\cup\pi_2$. Assume further that no prime in $\pi_1$ is adjacent to any prime in $\pi_2$ in the character degree graph $\Delta(G)$.

It is a well-known consequence of the general theory of character degree graphs of solvable groups that such a configuration forces $G$ to have Fitting height two. More precisely, up to isomorphism, $G$ may be assumed to have the form
\[
G=(A\times B)\rtimes P,
\]
where $A$ is an abelian $\pi_1$--group, $B$ is an abelian $\pi_2$--group, and $P$ is a nontrivial $p$--group acting faithfully and coprimely on $A\times B$, with nontrivial action on each direct factor. In this case, the Fitting subgroup of $G$ is given by $F(G)=A\times B$.

By Clifford theory, the irreducible character degrees of $G$ arise from linear characters and from characters induced from $A$ or $B$. Consequently, every nontrivial irreducible character degree of $G$ is divisible by $p$ together with primes from exactly one of the sets $\pi_1$ or $\pi_2$, and no irreducible character degree is divisible by primes from both $\pi_1$ and $\pi_2$. This accounts for the absence of edges between $\pi_1$ and $\pi_2$ in $\Delta(G)$, while the prime $p$ acts as a bridging vertex adjacent to both parts.
\end{rem}

In view of Remark~\ref{rem:general-structure-bridge}, solvable groups whose character degree graphs contain a single bridging prime naturally give rise to graphs of the form $K_1 \vee (K_{\ell_1}\sqcup K_{\ell_3})$. These graphs consist of two disjoint cliques connected only through one vertex, reflecting the separation of prime sets in $\rho(G)$. We now determine the adjacency spectrum and the associated splitting fields for this family of character degree graphs.

\begin{thm}\label{thm:splitting-two-cliques}
Let $n \ge 3$ be an integer and let $n_1$ satisfy
\[
1 \le n_1 \le n-(n_1+1).
\]
Set
\[
\ell_1 = n-(n_1+1), \qquad \ell_3 = n_1,
\]
so that $\ell_1,\ell_3 \ge 1$ and $\ell_1+\ell_3+1=n$. Suppose that $G$ is a finite solvable group whose character degree graph satisfies
\[
\Delta(G)\cong K_1 \,\vee\,\bigl(K_{\ell_1}\sqcup K_{\ell_3}\bigr),
\]
and let $A=A(\Delta(G))$ be the adjacency matrix. Then the following hold:
\begin{enumerate}[label=\upshape(\roman*)]
    \item The eigenvalue $-1$ occurs with multiplicity $n-3$.
    \item The remaining three eigenvalues are the roots of the cubic polynomial
    \[
    \chi(\lambda)
    =\lambda^3-(n-3)\lambda^2
      +\bigl(nn_1-2n-n_1^2-n_1+3\bigr)\lambda
      +\bigl(2nn_1-n-2n_1^2-2n_1+1\bigr),
    \]
    and hence
    \[
    \chi_A(\lambda)=(\lambda+1)^{\,n-3}\chi(\lambda).
    \]
    \item The polynomial $\chi(\lambda)$ is reducible over $\mathbb{Q}$ if and only if
    $n=2n_1+1$. In this case,
    \[
    \chi(\lambda)
    =\bigl(\lambda-(n_1-1)\bigr)
     \bigl(\lambda^2-(n_1-1)\lambda-2n_1\bigr),
    \]
    and the splitting field of $\chi(\lambda)$ over $\mathbb{Q}$ is a quadratic extension.
    \item If $n\neq 2n_1+1$, then $\chi(\lambda)$ is irreducible over $\mathbb{Q}$.
    \item If $n\neq 2n_1+1$, the splitting field of $\chi(\lambda)$ over $\mathbb{Q}$
    has degree $3$ when $\operatorname{Disc}(\chi(\lambda))$ is a perfect square in $\mathbb{Z}$,
    and degree $6$ otherwise.
\end{enumerate}
\end{thm}

\begin{proof}
We proceed in steps.

\medskip
\noindent\textbf{Step~1: Graph structure and equitable partition.}
The graph $\Delta(G)$ consists of two disjoint cliques $V_1\cong K_{\ell_1}$ and $V_3\cong K_{\ell_3}$ together with a universal vertex $v$ adjacent to every vertex in $V_1\cup V_3$. There are no edges between $V_1$ and $V_3$. Hence
\[
V(\Delta(G))=V_1\sqcup\{v\}\sqcup V_3
\]
is an equitable partition.

\medskip
\noindent\textbf{Step~2: Multiplicity of the eigenvalue $-1$.}
For each clique $K_{\ell_i}$ $(i=1,3)$, the vectors supported on that clique and summing to zero form an $(\ell_i-1)$--dimensional eigenspace with eigenvalue $-1$. These vectors vanish outside the clique and are unaffected by the universal vertex. Thus
\[
(\ell_1-1)+(\ell_3-1)=n-3,
\]
proving~(i).

\medskip
\noindent\textbf{Step~3: Reduction to a $3\times3$ matrix.}
Let $W\subseteq\mathbb{R}^n$ be the subspace of vectors constant on $V_1$, on $\{v\}$, and on $V_3$. With respect to the orthonormal basis
\[
\mathbf{u}_1=\frac{1}{\sqrt{\ell_1}}\mathbf{1}_{V_1},\qquad
\mathbf{u}_2=\mathbf{e}_v,\qquad
\mathbf{u}_3=\frac{1}{\sqrt{\ell_3}}\mathbf{1}_{V_3},
\]
the restriction of $A$ to $W$ is represented by
\[
\widehat{C}=
\begin{pmatrix}
\ell_1-1 & \sqrt{\ell_1} & 0\\
\sqrt{\ell_1} & 0 & \sqrt{\ell_3}\\
0 & \sqrt{\ell_3} & \ell_3-1
\end{pmatrix}.
\]
A direct computation of $\det(\lambda I-\widehat{C})$, together with $\ell_1=n-n_1-1$ and $\ell_3=n_1$, yields the cubic polynomial $\chi(\lambda)$. Since $\mathbb{R}^n=W\oplus W^\perp$ and $A$ preserves this decomposition,
\[
\chi_A(\lambda)=(\lambda+1)^{n-3}\chi(\lambda),
\]
establishing~(ii).

\medskip
\noindent\textbf{Step~4: Reducibility when $n=2n_1+1$.}
Evaluating $\chi(\lambda)$ at $\lambda=n_1-1$ gives
\[
\chi(n_1-1)=n_1(n-2n_1-1).
\]
Thus $\chi(n_1-1)=0$ if and only if $n=2n_1+1$, and in this case
\[
\chi(\lambda)
=\bigl(\lambda-(n_1-1)\bigr)
 \bigl(\lambda^2-(n_1-1)\lambda-2n_1\bigr).
\]
The discriminant of the quadratic factor is $n_1^2+6n_1+1$, which is not a perfect square for $n_1\ge1$. Hence the splitting field is a quadratic extension of $\mathbb{Q}$, proving~(iii).

\medskip
\noindent\textbf{Step~5: Irreducibility when $n\neq 2n_1+1$ (quotient--remainder method).}
Assume $n\neq 2n_1+1$. Since $\chi(\lambda)\in\mathbb{Z}[\lambda]$ is monic, any rational root must be an integer dividing the constant term
\[
c=2nn_1-n-2n_1^2-2n_1+1.
\]
Moreover, all eigenvalues of $A$ lie in the interval $[-(n-1),\,n-1]$. Hence any rational root $r$ must lie in the finite set
\[
\mathcal{R}=\{\,r\in\mathbb{Z} : r\mid c,\ |r|\le n-1\,\}.
\]

Let $r\in\mathcal{R}$. By the Remainder Theorem,
\[
\chi(\lambda)=(\lambda-r)q_r(\lambda)+\chi(r),
\]
where $\chi(r)$ is the remainder upon division by $(\lambda-r)$. Thus $\chi(r)=0$ if and only if $(\lambda-r)$ divides $\chi(\lambda)$.

For the two natural candidates suggested by the graph structure,
\[
r_1=n_1-1,\qquad r_2=n-n_1-2,
\]
a direct division yields
\[
\chi(r_1)=n_1(n-2n_1-1)\neq0,
\qquad
\chi(r_2)=(n-n_1-1)(n-2n_1-1)\neq0,
\]
since $n\neq 2n_1+1$. For any remaining $r\in\mathcal{R}$, direct substitution shows $\chi(r)\neq0$ as well. Hence no linear factor $(\lambda-r)$ with $r\in\mathbb{Q}$ divides $\chi(\lambda)$.

Therefore $\chi(\lambda)$ has no rational root. As $\chi(\lambda)$ is a monic cubic, it is irreducible over $\mathbb{Q}$, proving~(iv).

\medskip
\noindent\textbf{Step~6: Splitting fields.}
If $n\neq 2n_1+1$, then $\chi(\lambda)$ is an irreducible cubic over $\mathbb{Q}$. Its Galois group is cyclic of order $3$ if and only if the discriminant $\operatorname{Disc}(\chi(\lambda))$ is a perfect square in $\mathbb{Z}$, and is isomorphic to $S_3$ otherwise. Consequently, the splitting field has degree $3$ or $6$, establishing~(v). \hfill $\square$
\end{proof}

\begin{cor}\label{cor:degree-values}
All three possible degrees $2$, $3$, and $6$ occur for the splitting field of $\chi(\lambda)$:
\begin{enumerate}[label=\textup{(\alph*)}]
    \item $(n,n_1)=(3,1) \Rightarrow [K:\mathbb{Q}]=2$,
    \item $(n,n_1)=(10,1) \Rightarrow [K:\mathbb{Q}]=3$,
    \item $(n,n_1)=(5,1) \Rightarrow [K:\mathbb{Q}]=6$.
\end{enumerate}
Moreover, each graph $\Delta(G)$ in these cases is realizable by a finite solvable group.
\end{cor}

\begin{proof}
\textbf{(a) Degree 2.}  
For $(n,n_1) = (3,1)$, we have $n = 2n_1 + 1$. Then
\[
\chi(\lambda) = \lambda^3 - 2\lambda = \lambda(\lambda^2 - 2).
\]
The splitting field is $\mathbb{Q}(\sqrt{2})$, of degree $2$.

\textbf{(b) Degree 3.}  
For $(n,n_1) = (10,1)$, the cubic is
\[
\chi(\lambda) = \lambda^3 - 7\lambda^2 - 9\lambda + 7.
\]
It has no rational root (check $\pm1,\pm7$), so it is irreducible over $\mathbb{Q}$.

The discriminant of a monic cubic $\lambda^3 + a\lambda^2 + b\lambda + c$ is
\[
\Delta = a^2b^2 - 4b^3 - 4a^3c - 27c^2 + 18abc.
\]
Here $a = -7$, $b = -9$, $c = 7$. We compute:
\[
\begin{aligned}
a^2b^2 &= (-7)^2(-9)^2 = 49 \cdot 81 = 3969, \\
-4b^3 &= -4(-9)^3 = -4(-729) = 2916, \\
-4a^3c &= -4(-7)^3(7) = -4(-343)(7) = 9604, \\
-27c^2 &= -27(7)^2 = -27 \cdot 49 = -1323, \\
18abc &= 18(-7)(-9)(7) = 18 \cdot 441 = 7938.
\end{aligned}
\]
Summing these gives
\[
\Delta = 3969 + 2916 + 9604 - 1323 + 7938 = 23104.
\]
Now factor $23104$:
\[
23104 \div 16 = 1444, \qquad 1444 = 38^2,
\]
so
\[
23104 = 16 \cdot 38^2 = (4 \cdot 38)^2 = 152^2.
\]
Thus $\Delta = 152^2$ is a perfect square in $\mathbb{Z}$, so the Galois group is $C_3$ and the splitting field has degree $3$.

\textbf{(c) Degree 6.}  
For $(n,n_1) = (5,1)$, the cubic is
\[
\chi(\lambda) = \lambda^3 - 2\lambda^2 - 4\lambda + 2.
\]
It is irreducible over $\mathbb{Q}$ (no rational roots among $\pm1, \pm2$).

Using the same discriminant formula with $a = -2$, $b = -4$, $c = 2$, a direct symbolic computation  yields
\[
\Delta = 564.
\]
Factoring gives
\[
564  = 2^2 \cdot 3 \cdot 47.
\]
Since the primes $3$ and $47$ occur to an odd power, $564$ is not a perfect square in $\mathbb{Z}$. Hence the Galois group is $S_3$ and the splitting field has degree $6$. \hfill $\square$
\end{proof}

\begin{rem}\label{rem:existence-cases}
The three cases listed in Corollary~\ref{cor:degree-values} are each realizable by finite solvable groups. Indeed, in all three instances the associated graph has a cut vertex and a triangle-free complement, and hence arises as the character degree graph of a solvable group.
\begin{itemize}
    \item $(3,1)$: The graph $K_1 \vee (K_1 \sqcup K_1)$ has a cut vertex and a triangle-free complement.  
    By Lemma~\ref{lem:two-components}, there exists a solvable group whose character degree graph has two complete connected components of the required sizes, and an application of the direct product construction in Definition~\ref{direct product} yields the desired graph.
    \item $(10,1)$: The graph $K_1 \vee (K_8 \sqcup K_1)$ again has a cut vertex and a triangle-free complement.  
    Existence of a corresponding solvable group follows from Lemma~\ref{lem:two-components}, and the join structure is obtained via Definition~\ref{direct product}.
    \item $(5,1)$: The graph $K_1 \vee (K_3 \sqcup K_1)$ satisfies the same properties.  
    As above, Lemma~\ref{lem:two-components} guarantees the existence of a solvable group with the required disconnected structure, and Definition~\ref{direct product} produces the corresponding character degree graph.
\end{itemize}

Thus, each case in Corollary~\ref{cor:degree-values} corresponds to an actual solvable group, and the associated character degree graphs arise naturally from standard constructions.
\end{rem}

\section{Spectral Conditions and Bounds for Lewis Graphs}

Having defined Lewis graphs in the introduction, we now study their spectral structure.
\begin{thm}[Spectral Structure of Lewis Graphs]\label{thm:LewisSpectralCorrected}
Let $\Gamma$ be a Lewis graph with vertex partition
\[
V(\Gamma)=U\sqcup X\sqcup Y, \qquad |U|=n_1,\; |X|=n_2,\; |Y|=n_3,
\]
satisfying the following structural conditions:
\begin{enumerate}[label=\upshape(\roman*)]
    \item The induced subgraphs $\Gamma[U]$, $\Gamma[X]$, and $\Gamma[Y]$ are complete;
    \item Every vertex $u\in U$ is adjacent to all vertices in $X\cup Y$;
    \item For each $x\in X$, there is exactly one $y\in Y$ such that $xy\notin E(\Gamma)$, and this assignment is given by a surjective map $f\colon X\to Y$;
    \item Let $B_{XY}\in\{0,1\}^{n_2\times n_3}$ denote the \emph{bipartite defect matrix}, defined by $(B_{XY})_{x,y}=1$ if $xy\notin E(\Gamma)$ and $0$ otherwise, and let $r=\operatorname{rank}(B_{XY})$ over $\mathbb{Q}$.
\end{enumerate}
Let $A=A(\Gamma)$ be the adjacency matrix of $\Gamma$, and let $n = n_1+n_2+n_3$. Then the following hold:
\begin{enumerate}[label=\upshape(\alph*)]
    \item \textbf{(Guaranteed integer eigenvalues).}
    The eigenvalue $-1$ occurs with multiplicity at least
    \[
      \max\{0,\, n - 1 - 2r\}.
    \]
    \item \textbf{(Bound on exceptional eigenvalues).}
    At most $2r+1$ eigenvalues of $A$ differ from $-1$ (counting algebraic multiplicities).
    \item \textbf{(Degree of irreducible factors).}
    Every irreducible factor of the characteristic polynomial $\chi_A(\lambda)\in\mathbb{Q}[\lambda]$ corresponding to eigenvalues not equal to $-1$ has degree at most $2r+1$.
    \item \textbf{(Splitting-field bound).}
    Let $K$ be the splitting field of $\chi_A(\lambda)$ over $\mathbb{Q}$. Then
    \[
      [K:\mathbb{Q}] \leq (2r+1)^{\,2r+1}.
    \]
\end{enumerate}
\end{thm}

\begin{proof}
We proceed in four steps.

\medskip
\noindent\textbf{Step 1: Decomposition $A = A_0 + E$.}

Let $A_0$ be the adjacency matrix of the complete graph $K_n$, obtained from $\Gamma$ by adding all missing edges between $X$ and $Y$. Then
\[
A_0 = J_n - I_n,
\]
where $J_n$ is the all-ones matrix and $I_n$ is the identity. The spectrum of $A_0$ is
\[
\sigma(A_0) = \{\,n-1,\, -1^{(n-1)}\,\},
\]
so $A_0$ has eigenvalue $n-1$ with multiplicity $1$ and eigenvalue $-1$ with multiplicity $n-1$.

Define the defect matrix $E = A - A_0$. Ordering vertices as $(U,X,Y)$, we have
\[
E =
\begin{pmatrix}
0 & 0 & 0 \\
0 & 0 & -B_{XY} \\
0 & -B_{XY}^\top & 0
\end{pmatrix}.
\]
Since $B_{XY}$ has rank $r$, the block matrix $\begin{pmatrix} 0 & B_{XY} \\ B_{XY}^\top & 0 \end{pmatrix}$ has rank $2r$. Therefore
\[
\operatorname{rank}(E) = 2r.
\]

\medskip
\noindent\textbf{Step 2: The $-1$ eigenspace (Proof of part~\textup{(a)}).}

Let $\mathbf{1}\in\mathbb{R}^n$ denote the all-ones vector. Since $A_0 = J_n - I_n$, we have
\[
A_0 \mathbf{1} = (n-1)\mathbf{1}, \qquad \text{and} \qquad A_0 v = -v \;\text{ for all } v \perp \mathbf{1}.
\]
Define the subspace
\[
W = \mathbf{1}^\perp \cap \ker(E) \subseteq \mathbb{R}^n.
\]
For any $v \in W$, we have $v \perp \mathbf{1}$ and $Ev = 0$, so
\[
Av = (A_0 + E)v = A_0 v + 0 = -v.
\]
Thus every vector in $W$ is an eigenvector of $A$ with eigenvalue $-1$.

By the rank-nullity theorem, $\dim(\ker(E)) = n - \operatorname{rank}(E) = n - 2r$. Since $\mathbf{1}^\perp$ has codimension $1$ in $\mathbb{R}^n$, we obtain
\[
\dim(W) \geq \dim(\ker(E)) - 1 = (n - 2r) - 1 = n - 1 - 2r.
\]
Since multiplicities are non-negative, the eigenvalue $-1$ occurs with multiplicity at least $\max\{0, n - 1 - 2r\}$, proving part~\textup{(a)}.

\medskip
\noindent\textbf{Step 3: The exceptional subspace (Proof of parts~\textup{(b)} and~\textup{(c)}).}

Let $S = \operatorname{span}\{\mathbf{1}\} + \operatorname{im}(E)$. Since $\operatorname{im}(E)$ has dimension $\operatorname{rank}(E) = 2r$, we have
\[
\dim(S) \leq 1 + 2r.
\]
We claim that $S$ is $A$-invariant. Indeed:
\begin{itemize}
    \item $A\mathbf{1} = A_0\mathbf{1} + E\mathbf{1} = (n-1)\mathbf{1} + E\mathbf{1} \in S$, since $E\mathbf{1} \in \operatorname{im}(E) \subseteq S$;
    \item For any $w \in \operatorname{im}(E)$, write $w = Ez$. Then
    \[
    Aw = (A_0 + E)Ez = A_0 Ez + E^2 z.
    \]
    Since $A_0 = J_n - I_n$ and $J_n E = \mathbf{1}(\mathbf{1}^\top E)$, we have $A_0 Ez \in \operatorname{span}\{\mathbf{1}\} + \operatorname{im}(E) = S$. Also $E^2 z \in \operatorname{im}(E) \subseteq S$. Thus $Aw \in S$.
\end{itemize}
Therefore $S$ is $A$-invariant, and $\mathbb{R}^n = S \oplus S^\perp$ is an $A$-invariant orthogonal decomposition.

On $S^\perp$, we have $S^\perp \subseteq \mathbf{1}^\perp \cap \ker(E) = W$, so $A|_{S^\perp} = -I$. Hence all eigenvalues of $A$ not equal to $-1$ arise from the restriction $A|_S$, which is a matrix of size at most $(2r+1) \times (2r+1)$.

This proves:
\begin{itemize}
    \item At most $\dim(S) \leq 2r+1$ eigenvalues differ from $-1$ (part~\textup{(b)});
    \item Every irreducible factor of $\chi_A(\lambda)$ corresponding to eigenvalues $\neq -1$ divides the characteristic polynomial of $A|_S$, hence has degree at most $2r+1$ (part~\textup{(c)}).
\end{itemize}

\medskip
\noindent\textbf{Step 4: Splitting field bound (Proof of part~\textup{(d)}).}

Let $p(\lambda) \in \mathbb{Q}[\lambda]$ be the characteristic polynomial of $A|_S$. Then $\deg(p) = d \leq 2r+1$, and the splitting field $K$ of $\chi_A(\lambda)$ coincides with the splitting field of $p(\lambda)$ (since all other roots are $-1$).

The Galois group of $p(\lambda)$ over $\mathbb{Q}$ embeds into the symmetric group $S_d$, so
\[
[K:\mathbb{Q}] \leq |S_d| = d! \leq d^{\,d} \leq (2r+1)^{\,2r+1},
\]
proving part~\textup{(d)}.

\medskip
This completes the proof of Theorem~\ref{thm:LewisSpectralCorrected}. \hfill $\square$
\end{proof}

\begin{rem}[On Integrality and the Eigenvalue $0$]\label{rem:integrality-zero}
\begin{enumerate}
    \item \textbf{Integrality.} 
    We originally hoped we could get integrality for symmetric matchings, but this is false in general. For example, if $n_1=1$ and $n_2=n_3=1$, the graph is the path $P_3$, with eigenvalues $\{ \sqrt{2}, 0, -\sqrt{2} \}$. Integrality typically holds only when $n_1=0$ (no universal vertices), in which case $A = J_n - I_n - P$ where $P$ is the matching permutation. Since $J, I, P$ commute, the spectrum is integral. For $n_1 > 0$, the coupling between $U$ and the matched pairs can introduce quadratic irrationals.
    \item \textbf{Eigenvalue $0$.} 
    Theorem~\ref{thm:LewisSpectralCorrected} does not assert that $0$ is always an eigenvalue when $B_{XY}$ is singular. However, in many natural Lewis graphs arising from character degree graphs (for example, when $n_1=0$ and the defect matrix $B_{XY}$ has a nontrivial kernel orthogonal to the all-ones vector), one can construct nonzero vectors $v$ with $Av=0$, showing that the adjacency matrix $A$ is singular. A complete characterization of when $0$ occurs as an eigenvalue in terms of the defect matrix $B_{XY}$ is an interesting refinement, but is not needed for the spectral and field-degree bounds established above.
\end{enumerate}
\end{rem}

In our earlier work~\cite{Sivanesan2026}, we showed that $(n\!-\!2)$-regular character degree graphs of solvable groups are closed under edge addition: every supergraph obtained in this way is again realizable as the character degree graph of a solvable group of Fitting height two. Motivated by this structural robustness, we move beyond spectral values and study the arithmetic nature of the spectrum itself. In particular, we examine the splitting fields over $\mathbb{Q}$ generated by adjacency eigenvalues, and show in the following corollary that these fields are always quadratic.

\begin{cor}\label{cor:quadratic}
Let $n \ge 4$ be an even integer, and let $G$ be a finite solvable group such that its character degree graph $\Delta(G)$ is an $(n-2)$-regular graph on $n$ vertices. Construct a supergraph $\widetilde{\Delta}$ of $\Delta(G)$ by selecting $n_1$ disjoint non-edges of $\Delta(G)$ and adding them as edges, where $1 \le n_1 < \frac{n}{2}$. Then the field extension over $\mathbb{Q}$ generated by the eigenvalues of $\widetilde{\Delta}$ has degree $2$.
\end{cor}

\begin{proof}
Since $\Delta(G)$ is $(n-2)$-regular on $n$ vertices, its complement is $1$-regular. Because $n$ is even, the complement must be a perfect matching
\[
M=\bigl\{\{u_1,v_1\},\dots,\{u_{n/2},v_{n/2}\}\bigr\}.
\]
Hence
\[
\Delta(G)=K_n \setminus M,
\]
the complete graph on $n$ vertices with the edges of $M$ removed. 

Fix $n_1$ with $1\le n_1<\frac{n}{2}$, and let
\[
S=\{u_1,v_1,\dots,u_{n_1},v_{n_1}\}, \qquad |S|=2n_1.
\]
The only missing edges inside $S$ are $\{u_i,v_i\}$ for $1\le i \le n_1$. Adding these edges produces a graph $\widetilde{\Delta}$ in which:
\[
\widetilde{\Delta}[S]\cong K_{2n_1},
\]
and the induced subgraph on the remaining $m=n-2n_1$ vertices is
\[
G_1 = K_m \setminus M',
\]
where $M'$ is the restriction of the perfect matching $M$ to $V\setminus S$. This is the cocktail party graph on $m$ vertices (i.e., the complete graph minus a perfect matching). Moreover, every vertex of $S$ is adjacent to every vertex of $V\setminus S$. Thus
\[
\widetilde{\Delta}=G_1 \vee K_{2n_1}.
\]

\medskip
\noindent\textbf{Spectra of the components.}
The spectrum of $K_{2n_1}$ is
\[
\{\,2n_1-1,\;(-1)^{(2n_1-1)}\,\},
\]
and the spectrum of the cocktail party graph $G_1$ on $m$ vertices is
\[
\{\,m-2,\;0^{(m/2)},\;(-2)^{(m/2-1)}\,\}.
\]

\medskip
\noindent\textbf{Join spectrum.}
By the standard spectral formula for the join of two regular graphs (see Theorem~\ref{1.1}), the spectrum of $\widetilde{\Delta}$ consists of:
\begin{itemize}
    \item all eigenvalues of $G_1$ except its largest eigenvalue $m-2$, namely $0^{(m/2)}$ and $(-2)^{(m/2-1)}$,
    \item all eigenvalues of $K_{2n_1}$ except its largest eigenvalue $2n_1-1$, namely $(-1)^{(2n_1-1)}$,
    \item the two eigenvalues of the coupling matrix
    \[
    \widehat{C}=
    \begin{pmatrix}
        m-2 & \sqrt{m\cdot 2n_1} \\
        \sqrt{m\cdot 2n_1} & 2n_1-1
    \end{pmatrix}.
    \]
\end{itemize}
All eigenvalues listed above, except the last two, are integers and hence lie in $\mathbb{Q}$.

A direct computation shows that the characteristic polynomial of $\widehat{C}$ is
\[
\chi(\lambda)=\lambda^2-(n-3)\lambda-(n+2n_1-2).
\]

\medskip
\noindent\textbf{Discriminant calculation.}
The discriminant of $\chi(\lambda)$ is
\[
D=(n-3)^2+4(n+2n_1-2)
    =(n-1)^2+8n_1.
\]
Suppose, for contradiction, that $D=k^2$ for some integer $k$. Then
\[
k^2-(n-1)^2 = 8n_1 \quad\Longrightarrow\quad (k-(n-1))(k+(n-1)) = 8n_1.
\]
Since $n$ is even, $n-1$ is odd, and $D \equiv 1 \pmod{8}$, so $k$ is odd. Hence both factors on the left are even. Write $k-(n-1)=2a$ and $k+(n-1)=2b$ with integers $b>a>0$. Then
\[
4ab = 8n_1 \;\Rightarrow\; ab = 2n_1, \qquad\text{and}\qquad b-a = n-1.
\]
Substituting $b = a + n - 1$ gives
\[
a(a + n - 1) = 2n_1.
\]
But $a \ge 1$ implies $a(a + n - 1) \ge n$, while $n_1 < \frac{n}{2}$ implies $2n_1 < n$. This contradiction shows that $D$ is not a perfect square.

\medskip
\noindent\textbf{Conclusion.}
Therefore, $\chi(\lambda)$ is irreducible over $\mathbb{Q}$, and its roots generate the quadratic field $\mathbb{Q}(\sqrt{D})$. Since all other eigenvalues of $\widetilde{\Delta}$ are rational, the field extension of $\mathbb{Q}$ generated by the full spectrum of $\widetilde{\Delta}$ is exactly $\mathbb{Q}(\sqrt{D})$, which has degree $2$ over $\mathbb{Q}$. \hfill $\square$
\end{proof}

\begin{exam}\label{ex:4vertex}
We illustrate Corollary~\ref{cor:quadratic} with the smallest nontrivial example. Consider a character degree graph $\Delta$ on four vertices and a supergraph $\widetilde{\Delta}$ obtained by adding a single missing edge. Figure~\ref{fig:four-vertex-cdg}(A) depicts the graph $\Delta$, which arises in solvable group constructions discussed by Lewis~\cite[p.~184]{Lewis2008}, while Figure~\ref{fig:four-vertex-cdg}(B) shows the corresponding supergraph $\widetilde{\Delta}$.

\begin{figure}[ht]
    \centering 
    \begin{subfigure}{.45\textwidth} 
        \centering 
        \begin{tikzpicture} 
            \draw[fill=black] (0,0) circle (2pt); 
            \draw[fill=black] (2,0) circle (2pt); 
            \draw[fill=black] (0,2) circle (2pt); 
            \draw[fill=black] (2,2) circle (2pt); 
            \draw[thick] (0,0)--(2,0)--(2,2)--(0,2)--(0,0);
        \end{tikzpicture} 
        \caption{$\Delta$}
        \label{fig:delta}
    \end{subfigure}
    \hfill
    \begin{subfigure}{.45\textwidth} 
        \centering 
        \begin{tikzpicture} 
            \draw[fill=black] (0,0) circle (2pt); 
            \draw[fill=black] (2,0) circle (2pt); 
            \draw[fill=black] (0,2) circle (2pt); 
            \draw[fill=black] (2,2) circle (2pt); 
            \draw[thick] (0,0)--(2,0)--(2,2)--(0,2)--(0,0);
            \draw[thick] (0,2)--(2,0);
        \end{tikzpicture} 
        \caption{$\widetilde{\Delta}$}
        \label{fig:delta-tilde}
    \end{subfigure}
    \caption{A character degree graph on four vertices and a supergraph}
    \label{fig:four-vertex-cdg}
\end{figure}
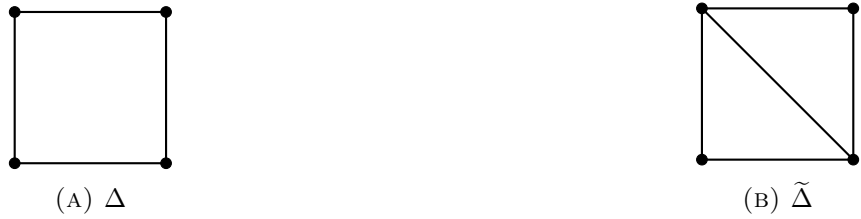

The graph $\widetilde{\Delta}$ admits an equitable partition with two cells, namely $G_1 = 2K_1$ and $G_2 = K_2$. With respect to this partition, the associated quotient matrix is
\[
Q =
\begin{pmatrix}
0 & 2 \\
2 & 1
\end{pmatrix}.
\]
Its characteristic polynomial is $\lambda^2 - \lambda - 4$, whose roots are
\[
\lambda = \frac{1 \pm \sqrt{17}}{2}.
\]

The remaining eigenvalues of $\widetilde{\Delta}$ arise from the internal structure of the cells: the component $G_1 = 2K_1$ contributes the eigenvalue $0$, while $G_2 = K_2$ contributes the eigenvalue $-1$. Hence the full adjacency spectrum of $\widetilde{\Delta}$ is
\[
\left\{ \frac{1+\sqrt{17}}{2},\; \frac{1-\sqrt{17}}{2},\; 0,\; -1 \right\}.
\]

The two irrational eigenvalues generate the quadratic field $\mathbb{Q}(\sqrt{17})$, whereas the remaining eigenvalues are rational. Consequently, the field generated by the full spectrum of $\widetilde{\Delta}$ is $\mathbb{Q}(\sqrt{17})$, a degree-$2$ extension of $\mathbb{Q}$, in complete agreement with Corollary~\ref{cor:quadratic}.
\end{exam}

\section{The Quotient Matrix and Characteristic Polynomial of a Character Degree Graph of Diameter \(3\)}
\label{sec:quotient-matrix}

Let \(G\) be a finite solvable group such that its character degree graph \(\Delta(G)\) has diameter \(3\), and suppose that \(\Delta(G)\) admits Lewis' partition
\[
\rho(G)=\rho_1\sqcup\rho_2\sqcup\rho_3\sqcup\rho_4.
\]
Assume that every \(\rho_i\) induces a complete subgraph, every vertex of \(\rho_i\) is adjacent to every vertex of \(\rho_{i+1}\) for \(1\leq i\leq 3\), and there are no edges between nonconsecutive parts. In particular, the edges joining \(\rho_2\) and \(\rho_3\) form a complete bipartite graph with bipartition \((\rho_2,\rho_3)\).

Set
\[
|\rho_1|=r_1,\qquad |\rho_2|=r_2,\qquad |\rho_3|=r_3,\qquad |\rho_4|=r_4.
\]
Using the reverse order \(\rho_4,\rho_3,\rho_2,\rho_1\), the graph is the
\(H\)-join
\[
\Delta(G)
\cong
P_4[K_{r_4},K_{r_3},K_{r_2},K_{r_1}].
\]

The spectrum of the associated graph is given by
\[
\sigma(\Delta(G))
=
\left\{-1^{(r_1+r_2+r_3+r_4-4)}\right\}
\cup
\sigma(\widehat C),
\]
where \(\widehat C\) is the \(4\times4\) matrix
\[
\widehat C=
\begin{pmatrix}
r_4-1 & \sqrt{r_4r_3} & 0 & 0\\
\sqrt{r_4r_3} & r_3-1 & \sqrt{r_3r_2} & 0\\
0 & \sqrt{r_3r_2} & r_2-1 & \sqrt{r_2r_1}\\
0 & 0 & \sqrt{r_2r_1} & r_1-1
\end{pmatrix}.
\]
Since \(\widehat C\) is real symmetric, all its eigenvalues are real.

The characteristic polynomial of \(\widehat C\) is the quartic polynomial studied in the next section.

\section{Irreducibility Criterion}
\label{sec:irreducibility-criterion}

\begin{thm}[Irreducibility in the diameter-three case]
\label{thm:diam3-irreducibility}
Under the preceding hypotheses, let
\[
p(\lambda)=\det(\lambda I_4-\widehat C)
=\lambda^4+c_3\lambda^3+c_2\lambda^2+c_1\lambda+c_0,
\]
where
\[
c_3=4-r_1-r_2-r_3-r_4,
\]
\[
c_2=
r_1r_3+r_1r_4+r_2r_4
-3r_1-3r_2-3r_3-3r_4+6,
\]
\[
\begin{aligned}
c_1={}&2r_1r_3+2r_1r_4+2r_2r_4
+r_1r_2r_3+r_2r_3r_4\\
&-3r_1-3r_2-3r_3-3r_4+4,
\end{aligned}
\]
and
\[
\begin{aligned}
c_0={}&r_1r_3+r_1r_4+r_2r_4
+r_1r_2r_3+r_2r_3r_4-r_1r_2r_3r_4\\
&-r_1-r_2-r_3-r_4+1.
\end{aligned}
\]
Then \(p(\lambda)\) is irreducible over \(\mathbb Q\) if and only if both of the following conditions hold:

\begin{enumerate}[label=\textup{(\roman*)}]
    \item
    \[
    c_0 \neq 0
    \]
    and
    \[
    p(e)\neq 0
    \qquad
    \text{for every } e\in\mathbb{Z} \text{ such that } e\mid c_0.
    \]
    \item There do not exist integers \(a,b,c,d\) satisfying
    \[
    \begin{aligned}
        a+c     &= c_3,\\
        ac+b+d  &= c_2,\\
        ad+bc   &= c_1,\\
        bd      &= c_0.
    \end{aligned}
    \]
    Equivalently, \(p(\lambda)\) cannot be expressed as
    \[
    p(\lambda)
    =
    (\lambda^2+a\lambda+b)
    (\lambda^2+c\lambda+d),
    \qquad a,b,c,d\in\mathbb{Z}.
    \]
\end{enumerate}

In particular, if \(r_1=r_2=1\), then
\[
\begin{aligned}
c_3&=2-r_3-r_4,\\
c_2&=-2r_3-r_4,\\
c_1&=r_3r_4+r_4-2,\\
c_0&=r_3+r_4-1,
\end{aligned}
\]
and hence
\[
\begin{aligned}
p(\lambda)
={}&\lambda^4+(2-r_3-r_4)\lambda^3-(2r_3+r_4)\lambda^2\\
&+(r_3r_4+r_4-2)\lambda+(r_3+r_4-1).
\end{aligned}
\]
In this special case, condition \textup{(ii)} is equivalent to the nonexistence of integers \(\alpha,\beta,\gamma,\delta\) satisfying
\[
\begin{aligned}
\alpha+\gamma&=2-r_3-r_4,\\
\alpha\gamma+\beta+\delta&=-2r_3-r_4,\\
\alpha\delta+\beta\gamma&=r_3r_4+r_4-2,\\
\beta\delta&=r_3+r_4-1.
\end{aligned}
\]

Moreover,
\[
p(-1)=-r_1r_2r_3r_4<0.
\]
Thus \(-1\) is never a root of \(p(\lambda)\). If \(r_1=2\), then
\[
p(1)=4r_2(r_3+r_4-2).
\]
In particular, if \(r_2\geq 1\), \(r_3\geq 3\), and \(r_4\geq 1\), then \(p(1)>0\), so \(1\) is not a root.
\end{thm}

\begin{proof}
Since \(p(\lambda)\) is monic and belongs to \(\mathbb Z[\lambda]\), Gauss's lemma implies that \(p(\lambda)\) is reducible over \(\mathbb Q\) if and only if it admits a nontrivial factorization into monic polynomials in \(\mathbb Z[\lambda]\). Because \(\deg p=4\), every such factorization has degree pattern \(1+3\) or \(2+2\).

\medskip
\noindent\textbf{Case 1: A linear--cubic factorization.}

If \(p(\lambda)\) has a linear factor over \(\mathbb Q\), then it has a rational root. Since \(p(\lambda)\) is monic with integer coefficients, every rational root is an integer. If \(e=0\) is a root, then \(c_0=p(0)=0\). If \(e\neq 0\) is a root, the rational root theorem gives \(e\mid c_0\).

Conversely, if \(c_0=0\), then \(p(0)=0\), and hence \(\lambda\) divides \(p(\lambda)\). If \(c_0\neq 0\) and an integer divisor \(e\mid c_0\) satisfies \(p(e)=0\), then the factor theorem shows that \(\lambda-e\) divides \(p(\lambda)\). Therefore, \(p(\lambda)\) has no linear factor over \(\mathbb Q\) if and only if condition \textup{(i)} holds.

\medskip
\noindent\textbf{Case 2: A quadratic--quadratic factorization.}

Suppose that \(p(\lambda)\) is a product of two quadratic polynomials over \(\mathbb Q\). Since \(p(\lambda)\) is monic and primitive, Gauss's lemma allows the factorization to be written as
\[
p(\lambda)
=
(\lambda^2+a\lambda+b)(\lambda^2+c\lambda+d),
\qquad a,b,c,d\in\mathbb Z.
\]
Expanding gives
\[
p(\lambda)
=
\lambda^4+(a+c)\lambda^3+(ac+b+d)\lambda^2+(ad+bc)\lambda+bd.
\]
Comparison of coefficients gives
\[
a+c=c_3,\qquad
ac+b+d=c_2,\qquad
ad+bc=c_1,\qquad
bd=c_0.
\]
Thus every quadratic--quadratic factorization produces an integer solution of the system in condition \textup{(ii)}.

Conversely, every integer solution of that system gives the displayed quadratic--quadratic factorization by direct expansion. Hence no such factorization exists if and only if condition \textup{(ii)} holds.

The two possible degree patterns have now been exhausted. Therefore, \(p(\lambda)\) is irreducible over \(\mathbb Q\) if and only if conditions \textup{(i)} and \textup{(ii)} both hold.

Substituting \(r_1=r_2=1\) into the general coefficient formulas gives
\[
\begin{aligned}
c_3&=2-r_3-r_4,\\
c_2&=-2r_3-r_4,\\
c_1&=r_3r_4+r_4-2,\\
c_0&=r_3+r_4-1,
\end{aligned}
\]
which proves the asserted specialization.

Finally, direct substitution into the general polynomial gives
\[
p(-1)=-r_1r_2r_3r_4.
\]
If \(r_1=2\), direct substitution and simplification give
\[
p(1)=4r_2(r_3+r_4-2).
\]
The stated sign conclusions follow from the assumed bounds. \hfill $\square$
\end{proof}

\section{Examples}
\label{sec:examples}

The irreducibility statements in this section are consequences of Theorem~\ref{thm:diam3-irreducibility}; in particular, all specialized polynomials below use the constant term
\[
c_0=r_3+r_4-1
\]
for \(r_1=r_2=1\).

\begin{exam}\label{ExPrime}
Let \(q\) be any odd prime number. Following the notation of Theorem~\ref{thm:diam3-irreducibility}, if
\[
r_1=r_2=1,\qquad r_3=q,\qquad r_4=1,
\]
then \(p(\lambda)\) is irreducible over \(\mathbb Q\).
\end{exam}

\begin{proof}
Since \(q\) is odd, we have \(q\geq 3\). For the parameters
\[
r_1=r_2=1,\qquad r_3=q,\qquad r_4=1,
\]
the polynomial in Theorem~\ref{thm:diam3-irreducibility} becomes
\[
p(\lambda)
=
\lambda^4+(1-q)\lambda^3-(2q+1)\lambda^2+(q-1)\lambda+q.
\]

We verify conditions \textup{(i)} and \textup{(ii)} of
Theorem~\ref{thm:diam3-irreducibility}.

\medskip
\noindent\textbf{Condition \textup{(i): no integer root.}}

The constant term is \(c_0=q\neq 0\), so the only possible integer roots are
\[
\pm 1,\qquad \pm q,
\]
all of which are odd. Reducing \(p(d)\) modulo \(2\), we get
\[
p(d)\equiv d^4+d^2+1 \pmod{2},
\]
because \(1-q\) and \(q-1\) are even, while \(-(2q+1)\) and \(q\) are odd. Since \(d\) is odd, \(d^2\equiv d^4\equiv 1\pmod{2}\), and therefore
\[
p(d)\equiv 1+1+1\equiv 1\pmod{2}.
\]
Thus \(p(d)\neq 0\) for every integer divisor \(d\) of \(q\). Hence condition \textup{(i)} holds.

\medskip
\noindent\textbf{Condition \textup{(ii): no quadratic--quadratic factorization.}}

Suppose, for the sake of contradiction, that
\[
p(\lambda)
=
(\lambda^2+\alpha\lambda+\beta)
(\lambda^2+\gamma\lambda+\delta)
\]
for some integers \(\alpha,\beta,\gamma,\delta\). Comparing coefficients gives
\[
\begin{aligned}
\alpha+\gamma&=1-q,\\
\alpha\gamma+\beta+\delta&=-(2q+1),\\
\alpha\delta+\beta\gamma&=q-1,\\
\beta\delta&=q.
\end{aligned}
\]

Since \(q\) is prime, the possible ordered pairs \((\beta,\delta)\) are
\[
(1,q),\qquad(q,1),\qquad(-1,-q),\qquad(-q,-1).
\]
In every case we substitute \(\gamma=1-q-\alpha\) into the third equation and then test the second equation.

If \((\beta,\delta)=(1,q)\), then
\[
q\alpha+\gamma=q-1.
\]
Thus
\[
q\alpha+1-q-\alpha=q-1,
\]
so
\[
(q-1)\alpha=2(q-1),
\]
whence \(\alpha=2\) and \(\gamma=-q-1\). But then
\[
\alpha\gamma+\beta+\delta
=
2(-q-1)+1+q
=
-q-1
\neq
-2q-1,
\]
a contradiction.

If \((\beta,\delta)=(q,1)\), then
\[
\alpha+q\gamma=q-1.
\]
Using \(\gamma=1-q-\alpha\), we obtain
\[
\alpha+q(1-q-\alpha)=q-1,
\]
so
\[
(1-q)\alpha=q^2-1.
\]
Hence \(\alpha=-q-1\) and \(\gamma=2\). Again
\[
\alpha\gamma+\beta+\delta
=
(-q-1)(2)+q+1
=
-q-1
\neq
-2q-1,
\]
a contradiction.

If \((\beta,\delta)=(-1,-q)\), then
\[
-q\alpha-\gamma=q-1.
\]
Thus
\[
-q\alpha-(1-q-\alpha)=q-1,
\]
so
\[
(1-q)\alpha=0,
\]
whence \(\alpha=0\) and \(\gamma=1-q\). Then
\[
\alpha\gamma+\beta+\delta
=
0-1-q
=
-q-1
\neq
-2q-1,
\]
a contradiction.

If \((\beta,\delta)=(-q,-1)\), then
\[
-\alpha-q\gamma=q-1.
\]
Using \(\gamma=1-q-\alpha\), we obtain
\[
-\alpha-q(1-q-\alpha)=q-1,
\]
so \(\alpha=1-q\) and \(\gamma=0\). Then
\[
\alpha\gamma+\beta+\delta
=
0-q-1
=
-q-1
\neq
-2q-1,
\]
a contradiction.

Thus no such integers \(\alpha,\beta,\gamma,\delta\) exist, and condition \textup{(ii)} holds. Since both conditions of Theorem~\ref{thm:diam3-irreducibility} are satisfied, \(p(\lambda)\) is irreducible over \(\mathbb Q\). \hfill $\square$
\end{proof}

See Figures~\ref{FigEx3} and~\ref{FigEx5} below for two examples of such graphs. As such, these are cataloged in \cite[p.~629]{Lewis2002} and \cite[p.~14]{Lewis2026}, respectively.

\begin{figure}[htb]
\centering
\begin{tikzpicture}[scale=3]
\node[inner sep=0pt] (a1) at (0,.5) {$\bullet$};
\node[inner sep=0pt] (a2) at (.5,.5) {$\bullet$};
\node[inner sep=0pt] (a3) at (1,.5) {$\bullet$};
\node[inner sep=0pt] (a4) at (1.5,1) {$\bullet$};
\node[inner sep=0pt] (a5) at (1.5,0) {$\bullet$};
\node[inner sep=0pt] (a6) at (2,.5) {$\bullet$};

\draw
(a1)--(a2)
(a2)--(a3)
(a2)--(a4)
(a2)--(a5)
(a3)--(a4)
(a3)--(a5)
(a3)--(a6)
(a4)--(a5)
(a4)--(a6)
(a5)--(a6);
\end{tikzpicture}
\caption{Example with \(r_3=3\) and \(r_4=1\).}
\label{FigEx3}
\end{figure}
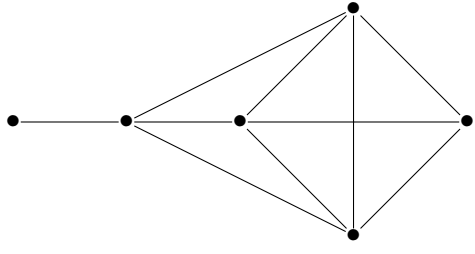

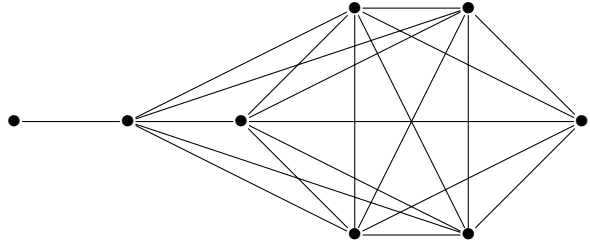
\begin{figure}[htb]
\centering
\begin{tikzpicture}[scale=3]
\node[inner sep=0pt] (c1) at (0,.5) {$\bullet$};
\node[inner sep=0pt] (c2) at (.5,.5) {$\bullet$};
\node[inner sep=0pt] (c3) at (1,.5) {$\bullet$};
\node[inner sep=0pt] (c4) at (1.5,1) {$\bullet$};
\node[inner sep=0pt] (c5) at (1.5,0) {$\bullet$};
\node[inner sep=0pt] (c6) at (2,1) {$\bullet$};
\node[inner sep=0pt] (c7) at (2,0) {$\bullet$};
\node[inner sep=0pt] (c8) at (2.5,.5) {$\bullet$};

\draw
(c1)--(c2)
(c2)--(c3)
(c2)--(c4)
(c2)--(c5)
(c2)--(c6)
(c2)--(c7)
(c3)--(c4)
(c3)--(c5)
(c3)--(c6)
(c3)--(c7)
(c3)--(c8)
(c4)--(c5)
(c4)--(c6)
(c4)--(c7)
(c4)--(c8)
(c5)--(c6)
(c5)--(c7)
(c5)--(c8)
(c6)--(c7)
(c6)--(c8)
(c7)--(c8);
\end{tikzpicture}
\caption{Example with \(r_3=5\) and \(r_4=1\).}
\label{FigEx5}
\end{figure}

The family in Example~\ref{ExPrime} has \(r_4=1\). Irreducible examples also occur when both \(r_3\) and \(r_4\) are odd, but then one must still verify the quadratic-factor condition in Theorem~\ref{thm:diam3-irreducibility} separately. Here is one such example.

\begin{exam}\label{ExOdd}
If
\[
r_1=r_2=1,\qquad r_3=3,\qquad r_4=3,
\]
then \(p(\lambda)\) is irreducible over \(\mathbb Q\).
\end{exam}

\begin{proof}
For these parameters, the polynomial in Theorem~\ref{thm:diam3-irreducibility} is
\[
p(\lambda)
=
\lambda^4-4\lambda^3-9\lambda^2+10\lambda+5.
\]

We again verify conditions \textup{(i)} and \textup{(ii)} of
Theorem~\ref{thm:diam3-irreducibility}.

\medskip
\noindent\textbf{Condition \textup{(i): no integer root.}}

The constant term is \(5\), so its only integer divisors are
\[
\pm 1,\qquad \pm 5,
\]
all of which are odd. Reducing modulo \(2\), we have
\[
p(d)\equiv d^4+d^2+1\pmod{2}.
\]
For every odd integer \(d\),
\[
d^2\equiv d^4\equiv 1\pmod{2},
\]
so
\[
p(d)\equiv 1+1+1\equiv 1\pmod{2}.
\]
Thus \(p(d)\neq 0\) for every integer divisor \(d\) of \(5\).

\medskip
\noindent\textbf{Condition \textup{(ii): no quadratic--quadratic factorization.}}

Suppose
\[
p(\lambda)
=
(\lambda^2+\alpha\lambda+\beta)
(\lambda^2+\gamma\lambda+\delta)
\]
with \(\alpha,\beta,\gamma,\delta\in\mathbb Z\). Comparing coefficients gives
\[
\begin{aligned}
\alpha+\gamma&=-4,\\
\alpha\gamma+\beta+\delta&=-9,\\
\alpha\delta+\beta\gamma&=10,\\
\beta\delta&=5.
\end{aligned}
\]

Since \(5\) is prime, the only possible ordered pairs \((\beta,\delta)\) are
\[
(1,5),\quad (5,1),\quad (-1,-5),\quad (-5,-1).
\]
We check each case, substituting \(\gamma=-4-\alpha\) into the third
equation.

If \((\beta,\delta)=(1,5)\), then
\[
5\alpha+\gamma=10.
\]
Thus
\[
5\alpha-4-\alpha=10,
\]
so
\[
4\alpha=14,
\]
which has no integer solution.

If \((\beta,\delta)=(5,1)\), then
\[
\alpha+5\gamma=10.
\]
Thus
\[
\alpha+5(-4-\alpha)=10,
\]
so
\[
-4\alpha=30,
\]
which has no integer solution.

If \((\beta,\delta)=(-1,-5)\), then
\[
-5\alpha-\gamma=10.
\]
Thus
\[
-5\alpha-(-4-\alpha)=10,
\]
so
\[
-4\alpha=6,
\]
which has no integer solution.

If \((\beta,\delta)=(-5,-1)\), then
\[
-\alpha-5\gamma=10.
\]
Thus
\[
-\alpha-5(-4-\alpha)=10,
\]
so
\[
4\alpha=-10,
\]
which has no integer solution.

Thus no quadratic--quadratic factorization exists. Therefore condition \textup{(ii)} holds, and \(p(\lambda)\) is irreducible over \(\mathbb Q\). \hfill $\square$
\end{proof}

Figure~\ref{FigEx33} graphically illustrates this example; see \cite[p.~14]{Lewis2026}.

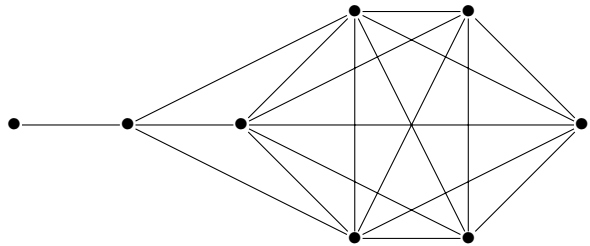
\begin{figure}[htb]
\centering
\begin{tikzpicture}[scale=3]
\node[inner sep=0pt] (d1) at (0,.5) {$\bullet$};
\node[inner sep=0pt] (d2) at (.5,.5) {$\bullet$};
\node[inner sep=0pt] (d3) at (1,.5) {$\bullet$};
\node[inner sep=0pt] (d4) at (1.5,1) {$\bullet$};
\node[inner sep=0pt] (d5) at (1.5,0) {$\bullet$};
\node[inner sep=0pt] (d6) at (2,1) {$\bullet$};
\node[inner sep=0pt] (d7) at (2,0) {$\bullet$};
\node[inner sep=0pt] (d8) at (2.5,.5) {$\bullet$};

\draw
(d1)--(d2)
(d2)--(d3)
(d2)--(d4)
(d2)--(d5)
(d3)--(d4)
(d3)--(d5)
(d3)--(d6)
(d3)--(d7)
(d3)--(d8)
(d4)--(d5)
(d4)--(d6)
(d4)--(d7)
(d4)--(d8)
(d5)--(d6)
(d5)--(d7)
(d5)--(d8)
(d6)--(d7)
(d6)--(d8)
(d7)--(d8);
\end{tikzpicture}
\caption{Example with \(r_3=3\) and \(r_4=3\).}
\label{FigEx33}
\end{figure}

\begin{rem}\label{rem:parity}
There is a useful parity observation for odd \(r_3\) and \(r_4\). Suppose \(r_1=r_2=1\) and both \(r_3\) and \(r_4\) are odd. Then the specialized polynomial in Theorem~\ref{thm:diam3-irreducibility} is
\[
p(\lambda)
=
\lambda^4+(2-r_3-r_4)\lambda^3
-(2r_3+r_4)\lambda^2
+(r_3r_4+r_4-2)\lambda
+(r_3+r_4-1).
\]
The constant term
\[
r_3+r_4-1
\]
is odd, so every integer divisor \(d\) of the constant term is odd. Moreover, the coefficients of \(\lambda^3\) and \(\lambda\) are even, while the coefficients of \(\lambda^2\) and the constant term are odd. Hence, for every odd integer \(d\),
\[
p(d)\equiv d^4+d^2+1\equiv 1\pmod{2}.
\]
Thus \(p(d)\neq 0\) for every integer divisor \(d\) of the constant term. This proves condition~\textup{(i)} of
Theorem~\ref{thm:diam3-irreducibility}.

However, this parity argument does not by itself prove condition~\textup{(ii)} of Theorem~\ref{thm:diam3-irreducibility}. One must still rule out a quadratic--quadratic factorization. Example~\ref{ExOdd} does this in the concrete case \(r_3=r_4=3\).
\end{rem}

\begin{rem}\label{rem:converse-false}
The converse of the preceding examples is not true. That is, the statement ``If \(p(\lambda)\) is irreducible over \(\mathbb Q\), then
\(r_1=r_2=1\) and \(r_3\) and \(r_4\) are both odd'' is false.

For example, take
\[
r_1=r_2=1,\qquad r_3=4,\qquad r_4=1.
\]
Then Theorem~\ref{thm:diam3-irreducibility} gives
\[
p(\lambda)
=
\lambda^4-3\lambda^3-9\lambda^2+3\lambda+4.
\]
This polynomial is irreducible over \(\mathbb Q\), even though \(r_3=4\) is
even.
\end{rem}

\begin{proof}
We verify conditions \textup{(i)} and \textup{(ii)} of
Theorem~\ref{thm:diam3-irreducibility}.

\medskip
\noindent\textbf{Condition \textup{(i): no linear factor.}}

Here \(c_0=4\), so the possible integer roots are
\[
\pm1,\qquad \pm2,\qquad \pm4.
\]
Direct evaluation gives
\[
p(\pm1)=-4,\qquad p(2)=-34,\qquad p(-2)=2,
\]
and
\[
p(4)=-64,\qquad p(-4)=296.
\]
All of these values are nonzero, so condition \textup{(i)} holds.

\medskip
\noindent\textbf{Condition \textup{(ii): no quadratic--quadratic factorization.}}

Suppose that
\[
p(\lambda)
=
(\lambda^2+\alpha\lambda+\beta)
(\lambda^2+\gamma\lambda+\delta),
\qquad
\alpha,\beta,\gamma,\delta\in\mathbb Z.
\]
Comparing coefficients gives
\[
\begin{aligned}
\alpha+\gamma&=-3,\\
\alpha\gamma+\beta+\delta&=-9,\\
\alpha\delta+\beta\gamma&=3,\\
\beta\delta&=4.
\end{aligned}
\]

Since \(\beta\delta=4\), the possible ordered pairs are
\[
(\beta,\delta)\in
\left\{
(1,4),(4,1),(2,2),
(-1,-4),(-4,-1),(-2,-2)
\right\}.
\]
We examine these possibilities.

If \((\beta,\delta)=(2,2)\), then the third equation gives
\[
2\alpha+2\gamma=3.
\]
This is impossible because the left-hand side is even, whereas the right-hand side is odd.

Similarly, if \((\beta,\delta)=(-2,-2)\), then
\[
-2\alpha-2\gamma=3,
\]
which is again impossible by parity.

If \((\beta,\delta)=(1,4)\), then
\[
4\alpha+\gamma=3.
\]
Together with \(\alpha+\gamma=-3\), this gives
\[
\alpha=2,\qquad \gamma=-5.
\]
However,
\[
\alpha\gamma+\beta+\delta
=
(2)(-5)+1+4
=
-5
\neq
-9.
\]

If \((\beta,\delta)=(4,1)\), then
\[
\alpha+4\gamma=3.
\]
Together with \(\alpha+\gamma=-3\), this gives
\[
\alpha=-5,\qquad \gamma=2.
\]
But
\[
\alpha\gamma+\beta+\delta
=
(-5)(2)+4+1
=
-5
\neq
-9.
\]

If \((\beta,\delta)=(-1,-4)\), then
\[
-4\alpha-\gamma=3.
\]
Together with \(\alpha+\gamma=-3\), this gives
\[
\alpha=0,\qquad \gamma=-3.
\]
Nevertheless,
\[
\alpha\gamma+\beta+\delta
=
(0)(-3)-1-4
=
-5
\neq
-9.
\]

Finally, if \((\beta,\delta)=(-4,-1)\), then
\[
-\alpha-4\gamma=3.
\]
Together with \(\alpha+\gamma=-3\), this gives
\[
\alpha=-3,\qquad \gamma=0.
\]
Again,
\[
\alpha\gamma+\beta+\delta
=
(-3)(0)-4-1
=
-5
\neq
-9.
\]

Thus the coefficient system has no solution in integers. Consequently, \(p(\lambda)\) has no quadratic--quadratic factorization over \(\mathbb Q\). Since it also has no linear factor, \(p(\lambda)\) is irreducible over \(\mathbb Q\). \hfill $\square$
\end{proof}


\begin{thebibliography}{99}


\bibitem{Bissler2019b}
M.~Bissler and J.~Laubacher, 
Classifying families of character degree graphs of solvable groups, 
\textit{Int. J. Group Theory} \textbf{8} (2019), no.~4, 37--46.

\bibitem{Bissler2019a}
M.~Bissler, J.~Laubacher, and M.~L. Lewis, 
Classifying character degree graphs with six vertices, 
\textit{Beitr. Algebra Geom.} \textbf{60} (2019), no.~3, 499--511.

\bibitem{Bissler2025}
M.~Bissler, J.~Laubacher, and M.~L. Lewis, 
A family of graphs that cannot occur as character degree graphs of solvable groups, \textit{Beitr. Algebra Geom.} \textbf{66} (2025), no.~3, 727--738.

\bibitem{Cameron2025}
P.~J. Cameron, G.~Sivanesan, C.~Selvaraj, T.~Tamizh~Chelvam, and J.~Laubacher, 
On the metric dimension of the character degree graph of a solvable group, 
\textit{Commun. Algebra} \textbf{54} (2026), no.~5, 2039--2049.

\bibitem{Cardoso2013a}
D.~M. Cardoso, M.~A. de Freitas, E.~A. Martins, and M.~Robbiano, 
Spectra of graphs obtained by a generalization of the join graph operation, 
\textit{Discrete Math.} \textbf{313} (2013), 733--741.

\bibitem{Cardoso2011}
D.~M.~Cardoso, I.~Gutman, E.~A.~Martins, and M.~Robbiano,
A generalization of Fiedler's lemma and some applications,
\textit{Linear Multilinear Algebra} \textbf{59} (2011), no.~8, 929--942.

\bibitem{Cardoso2013b}
D.~M. Cardoso, E.~A. Martins, M.~Robbiano, and O.~Rojo, 
Eigenvalues of a H-generalized join graph operation constrained by vertex subsets, 
\textit{Linear Algebra Appl.} \textbf{438} (2013), no.~8, 3278--3290.

\bibitem{Cvetkovic1995}
D.~Cvetković, M.~Doob, and H.~Sachs, 
\textit{Spectra of Graphs: Theory and Application}, 
Johann Ambrosius Barth Verlag, Heidelberg-Leipzig, 1995.

\bibitem{Cvetkovic2010}
D.~Cvetković, P.~Rowlinson, and S.~Simić, 
\textit{An Introduction to the Theory of Graph Spectra}, 
Cambridge University Press, Cambridge, 2010.

\bibitem{LJ}
S.~DeGroot, J.~Laubacher, and M.~Medwid, 
On prime character degree graphs occurring within a family of graphs (ii), 
\textit{Comm. Algebra} \textbf{50} (2022), no.~8, 3307--3319.

\bibitem{DummitFoote}
D.~S. Dummit and R.~M. Foote, 
\textit{Abstract Algebra}, 3rd ed., 
John Wiley \& Sons, Hoboken, NJ, 2004.

\bibitem{Ebrahimi2015}
M.~Ebrahimi, A.~Iranmanesh, and M.~A.~Hosseinzadeh, 
Hamiltonian character graphs, 
\textit{J. Algebra} \textbf{428} (2015), 54--66.

\bibitem{Ebrahimi2023}
M.~Ebrahimi, M.~Khatami, and Z.~Mirzaei, 
Regular character-graphs whose eigenvalues are greater than or equal to $-2$, 
\textit{Discrete Math.} \textbf{346} (2023), no.~1, 113137.

\bibitem{Fiedler1974}
M.~Fiedler, 
Eigenvalues of nonnegative symmetric matrices, 
\textit{Linear Algebra Appl.} \textbf{9} (1974), 119--142.

\bibitem{Godsil2001}
C.~Godsil and G.~Royle, 
\textit{Algebraic Graph Theory}, 
Springer-Verlag, New York, 2001.

\bibitem{Harary1974}
F.~Harary and A.~J.~Schwenk, 
Which graphs have integral spectra?, 
in \textit{Graph Theory and Applications}, 
Springer, Berlin, 1974, pp.~45--51.

\bibitem{Hafezieh2021}
R.~Hafezieh, M.~A. Hosseinzadeh, S.~Hossein-zadeh, and A.~Iranmanesh, 
On cut vertices and eigenvalues of character graphs of solvable groups, 
\textit{Discrete Appl. Math.} \textbf{303} (2021), 86--93.

\bibitem{LM}
J.~Laubacher and M.~Medwid, 
On prime character degree graphs occurring within a family of graphs, 
\textit{Comm. Algebra} \textbf{49} (2021), no.~4, 1534--1547.

\bibitem{LMS}
J.~Laubacher, M.~Medwid, and D.~Schuster,
Classifying character degree graphs with seven vertices,
\textit{Adv. Group Theory Appl.} \textbf{22} (2025), 79--121.

\bibitem{Lewis2001}
M.~L. Lewis,  
Solvable groups whose degree graphs have two connected components,  
\textit{J. Group Theory} \textbf{4} (2001), no.~3, 255--275.

\bibitem{Lewis2002}
M.~L. Lewis,  
A solvable group whose character degree graph has diameter 3, 
\textit{Proc. Amer. Math. Soc.} \textbf{130} (2002), no.~3, 625--630.

\bibitem{Lewis2006}
M.~L. Lewis, 
Character degree graphs of solvable groups of fitting height 2, 
\textit{Canad. Math. Bull.} \textbf{49} (2006), no.~1, 127--133.

\bibitem{Lewis2008}
M.~L. Lewis, 
An overview of graphs associated with character degrees and conjugacy class sizes in finite groups, 
\textit{Rocky Mountain J. Math.} \textbf{38} (2008), no.~1, 175--211.

\bibitem{LewisMeng2019}
M.~L. Lewis and Q.~Meng, 
Solvable groups whose prime divisor character degree graphs are 1-connected, 
\textit{Monatsh. Math.} \textbf{190} (2019), 541--548.

\bibitem{Lewis2026}
M.~L. Lewis and A.~Summers,
Classifying prime character degree graphs with eight vertices,
\textit{arXiv preprint} arXiv:2603.15851 (2026), 38 pp.

\bibitem{Manz1985}
O.~Manz, 
Degree problems II: $\pi$-separable character degrees, 
\textit{Comm. Algebra} \textbf{13} (1985), 2421--2431.

\bibitem{ManzStaszewski1988}
O.~Manz, R.~Staszewski, and W.~Willems, 
On the number of components of a graph related to character degrees, 
\textit{Proc. Amer. Math. Soc.} \textbf{103} (1988), no.~1, 31--37.

\bibitem{ManzWolf1989}
O.~Manz, W.~Willems, and T.~R. Wolf, 
The diameter of the character degree graph, 
\textit{J. Reine Angew. Math.} \textbf{402} (1989), 181--198.

\bibitem{Monius2022}
K.~M\"onius, 
Splitting fields of spectra of circulant graphs, 
\textit{J. Algebra} \textbf{594} (2022), 154--169.

\bibitem{Zuccari2014}
C.~P. Morresi Zuccari, 
Regular character degree graphs, 
\textit{J. Algebra} \textbf{411} (2014), 215--224.

\bibitem{Palfy1998}
P.~P. Pálfy, 
On the character degree graph of solvable groups I: three primes, 
\textit{Period. Math. Hungar.} \textbf{36} (1998), no.~1, 61--65.

\bibitem{Sivanesan2024}
G.~Sivanesan, C.~Selvaraj, and T.~Tamizh Chelvam, 
Eulerian character degree graphs of solvable groups, 
\textit{AKCE Int. J. Graphs Combin.} \textbf{21} (2024), no.~2, 161--166. 
doi:10.1080/09728600.2024.2309620.

\bibitem{Sivanesan2026}
G.~Sivanesan and C.~Selvaraj,
Laplacian eigenvalues of character degree graphs of solvable groups,
\textit{Vietnam J. Math.} \textbf{54} (2026), 117--134.
doi:10.1007/s10013-024-00705-y.

\bibitem{Wu2024}
Y.~Wu, Q.~Guo, J.~Yang, and L.~Feng, 
Splitting fields of some matrices of normal (mixed) Cayley graphs, 
\textit{Discrete Math.} \textbf{347} (2024), no.~5, 113914.


\end{thebibliography}
\end{document}